\documentclass[11pt,twoside]{amsart}
\usepackage[latin1]{inputenc}
\usepackage[OT1]{fontenc}
\usepackage{amsmath}
\usepackage{amsthm}
\usepackage{amssymb}
\usepackage[all]{xy}

\newtheorem{thm}{Theorem}[section]

\newtheorem{prop}[thm]{Proposition}

\newtheorem{lem}[thm]{Lemma}
\newtheorem{cor}[thm]{Corollary}

\numberwithin{equation}{section}

\theoremstyle{definition}
\newtheorem{definition}[thm]{Definition}
\newtheorem{remark}[thm]{Remark}
\newtheorem{ex}[thm]{Example}

\newcommand{\qqed}{\hspace*{\fill}$\Box$}

\newcommand{\im}{\operatorname{im}}
\newcommand{\Db}{{\rm D}^{\rm b}}

\newcommand{\Aut}{{\rm Aut}}

\newcommand{\End}{{\rm End}}
\newcommand{\Hom}{{\rm Hom}}
\newcommand{\Spec}{{\rm Spec}}

\newcommand{\Perf}{{\rm Perf}}

\newcommand{\id}{{\rm id}}

\newcommand{\ka}{{\mathcal A}}
\newcommand{\kb}{{\mathcal B}}
\newcommand{\kc}{{\mathcal C}}

\newcommand{\ki}{{\mathcal I}}

\newcommand{\ko}{{\mathcal O}}
\newcommand{\kp}{{\mathcal P}}

\newcommand{\kt}{{\mathcal T}}
\newcommand{\ku}{{\mathcal U}}

\newcommand{\IN}{\mathbb{N}}

\newcommand{\IQ}{\mathbb{Q}}

\newcommand{\IZ}{\mathbb{Z}}

\DeclareMathOperator{\RHom}{\textit{R}\mathcal{H}\textit{om}}

\renewcommand{\to}{\xymatrix@1@=15pt{\ar[r]&}}
\renewcommand{\rightarrow}{\xymatrix@1@=15pt{\ar[r]&}}
\renewcommand{\mapsto}{\xymatrix@1@=15pt{\ar@{|->}[r]&}}
\renewcommand{\twoheadrightarrow}{\xymatrix@1@=15pt{\ar@{->>}[r]&}}
\renewcommand{\hookrightarrow}{\xymatrix@1@=15pt{\ar@{^(->}[r]&}}
\newcommand{\congpf}{\xymatrix@1@=15pt{\ar[r]^-\sim&}}
\renewcommand{\cong}{\simeq}

\begin{document}

\title[Scalar extensions of triangulated categories]{Scalar extensions of triangulated categories}
\author[P.\ Sosna]{Pawel Sosna}

\address{Dipartimento di Matematica ``F.\ Enriques'', Universit\`a degli Studi di Milano, Via Cesare Saldini 50, 20133 Milan, Italy}
\email{pawel.sosna@guest.unimi.it}



\begin{abstract} \noindent
Given a triangulated category $\kt$ over a field $K$ and a field extension $L/K$, we investigate how one can construct a triangulated category $\kt_L$ over $L$. Our approach produces the derived category of the base change scheme $X_L$ if $\kt$ is the bounded derived category of a smooth projective variety over $K$ and the field extension is finite and Galois. We also investigate how the dimension of a triangulated category behaves under scalar extensions.

\vspace{-2mm}\end{abstract}
\maketitle

\maketitle
\let\thefootnote\relax\footnotetext{This work was supported by the grant SO 1095/1-1 and the SFB/TR 45 `Periods,
Moduli Spaces and Arithmetic of Algebraic Varieties' of the DFG
(German Research Foundation).}

\section{Introduction}
Base change techniques are ubiquitous in algebraic geometry. In particular, scalar extensions are often used to study properties of geometric objects.
In recent years it also became commonplace to study the geometry of a smooth projective variety $X$ over some field $K$ via the derived category $\Db(X)$ of its abelian category of coherent sheaves ${\rm{Coh}}(X)$. Under our assumptions $\Db(X)$ is a $K$-linear triangulated category, and in fact the derived category of an abelian category is one major source of examples of triangulated categories. We will use the example from geometry as a guide to introduce scalar extensions for triangulated categories. Thus, to any field extension $L/K$ and a $K$-linear triangulated category $\kt$ we would like to associate an $L$-linear triangulated category $\kt_L$. Of course, the construction should produce the expected result in the standard examples, e.g.\ if $\kt$ is $\Db(X)$ for $X$ as above, then the base change category $\kt_L$ should be equivalent to the derived category of $X_L=X\times_K L$. We will often assume that the field extension $L/K$ is finite, although some of the arguments do indeed generalise to arbitrary extensions.\smallskip

The problem one faces in proposing a reasonable construction is that triangulated categories are not as rigid as, say, abelian categories. For the latter categories, as well as for additive ones without additional structure, there is in fact a well-known and fairly simple construction (see e.g.\ \cite{AZ} or \cite{LvdB}), which gives the expected results if applied to e.g.\ the abelian category of (quasi-)coherent sheaves on a scheme $X$. This construction is recalled in detail in section two. There is also a slightly different approach which appears in \cite{Deligne} or in \cite{Stalder} and which is structurally similar, but uses Ind-objects. We can avoid this more technical construction, mostly because we usually work with finite extensions. The reason why this approach cannot work for a triangulated category basically boils down to the fact that the cone is not functorial. To circumvent this problem we shall work with triangulated categories arising as homotopy categories of pretriangulated DG-categories (in other words, triangulated categories admitting an enhancement). The advantage is that in the latter ones the cone is in fact functorial. In section three we will therefore recall the basic definitions and properties of (pretriangulated) differential graded categories and introduce scalar extensions for them.\smallskip

In section four we present the definition of base change: The basic idea would be to write $\kt$ as the homotopy category of a pretriangulated DG-category $\ka$, do base change for $\ka$ and consider the homotopy category of the base change category. However, this simple direct approach does not work and one has to make the definition slightly more involved. We then prove our
\smallskip

\noindent {\bf Main result 1 (Propositions \ref{smooth-bc} and \ref{noeth-bc})} {\it Given a triangulated category $\kt$ over $K$ which arises as the homotopy category of a pretriangulated DG-category, there is a natural way to define an $L$-linear triangulated category $\kt_L$. If $X$ is a smooth projective variety over $K$ and $\kt\cong\Db(X)$, then $\kt_L\cong\Db(X_L)$. If $L/K$ is finite, then the last statement holds for any Noetherian scheme $X$.}
\smallskip

Our construction a priori depends on the pretriangulated DG-category, but see Proposition \ref{ind-choice} for a partial result concerning independence of this choice. We conclude the section by sketching an alternative approach towards the definition of base change which uses the close connection between so-called algebraic triangulated categories (basically all examples in algebraic geometry and representation theory are of this type) and derived categories of DG-categories. This approach is probably more elegant but the DG-categories appearing in it are more difficult to describe. 
In the last section we consider the behaviour of the dimension of a triangulated category under base change. We prove
\smallskip

\noindent {\bf Main result 2 (Corollary \ref{dim-ab-bc})} {\it Let $\kc$ be an abelian category with enough injectives and with generators and let $L/K$ be a finite Galois extension. Then $\dim(\Db(\kc)_L)= \dim(\Db(\kc))$. In particular, $\dim(\Db(X_L))=\dim(\Db(X))$ for any Noetherian scheme $X$.}
\smallskip

\noindent{\bf Acknowledgements.} This paper is based on the last chapter of my PhD thesis \cite{PS} which was supervised by Daniel Huybrechts whom I would like to thank for a lot of fruitful discussions. I am grateful to Bernhard Keller and Valery Lunts for valuable suggestions. The final writeup of this article was done during my research stay at the Universit\`{a} degli Studi di Milano and I would like to thank the department of mathematics and the complex geometry group for their hospitality.

\section{Scalar extensions for additive categories}

\begin{definition}
Let $\kc$ be a $K$-linear additive category and let $L/K$ be a field extension. The \emph{base change category} $\kc_L$ is defined as follows:\newline
$\bullet$ Objects of $\kc_L$ are pairs $(C,f)$, where $C \in \kc$ and $f\colon L \rightarrow \End_\kc(C)$ is a morphism of $K$-algebras.\newline
$\bullet$ Morphisms between $(C,f)$ and $(D,g)$ are given by morphisms $\alpha\colon C \rightarrow D$ in $\kc$ compatible with the given actions of $L$, i.e.\ for any $l \in L$ the diagram 
\[\begin{xy}
\xymatrix{ C \ar[d]_{f(l)} \ar[r]^\alpha & C \ar[d]^{g(l)} \\
					 D \ar[r]^\alpha & D}  
\end{xy}\]
commutes.
\end{definition}

We call the datum $(C,f)$ an $L$-\emph{module structure} on $C$. 

\begin{lem}
The category $\kc_L$ is additive and comes with a natural $L$-linear structure.
\end{lem} 

\begin{proof}
The verification is straightforward: The zero object is $(0,0)$, the direct sum of $(C,f)$ and $(D,g)$ is given by $(C\oplus D, f\oplus g)$, the $K$-linearity is obvious. As to the $L$-linearity: For a scalar $l \in L$ and an $\alpha \in \Hom_{\kc_L}((C,f), (D,g))$ define $l\cdot \alpha:=\alpha \circ f(l)=g(l) \circ \alpha$. It is then easy to check that this is well-defined and thus $\kc_L$ is indeed $L$-linear.
\end{proof}

\begin{lem}
If $\kc$ is an abelian category, then $\kc_L$ is also abelian.
\end{lem}

\begin{proof}
Let $\alpha\colon (C,f) \rightarrow (D,g)$ be a morphism in $\kc_L$. We first have to show the existence of a kernel and a cokernel. We will show the existence of the former, the latter is similar. Forgetting the additional structures there exists a kernel $A$ in $\kc$. One can define a canonical morphism $h\colon L \rightarrow \End_\kc(A)$ as follows: Let $l \in L$ be arbitrary and consider the commutative diagram
\[
\begin{xy}
\xymatrix{ A \ar[r]^i & C \ar[r]^\alpha \ar[d]^{f(l)} & D \ar[d]^{g(l)} \\
					A \ar[r]^i & C \ar[r]^\alpha & D }
\end{xy}\] 
Since $\alpha\circ f(l) \circ i= g(l) \circ \alpha \circ i=0$, there exists a unique morphism $\begin{xy}\xymatrix{A \ar[r]^{h(l)} & A}\end{xy}$ making the diagram commutative. This defines $h$ and makes $i$ a morphism in $\kc_L$. The axiom about the equality of the image and the coimage is equally easy to check.
\end{proof}

Let us now consider base change for functors.

\begin{definition}
Let $F\colon \ka \rightarrow \kb$ be a functor between $K$-linear abelian (or additive) categories. The functor $F_L\colon \ka_L \rightarrow \kb_L$ is defined as follows: For an $L$-module $(A,f) \in \ka_L$ define a module structure $\widetilde{f}$ on $F(A)$ by the composition $L \rightarrow \End_\ka(A) \rightarrow \End_\kb(F(A))$, i.e.\ $\widetilde{f}(l)=F(f(l))$. For any $\alpha\colon (A,f) \rightarrow (A',g)$ the map $F(\alpha)$ is then compatible with the module structures on $F(A)$ and $F(A')$ and this defines $F$ on morphisms.  
\end{definition}
 
Note that with this definition $F_L$ is exact if $F$ is. Furthermore one has

\begin{lem}
If $F$ is an equivalence, then $F_L$ is also an equivalence.
\end{lem}

\begin{proof}
Let $(A,f)$ and $(A',g)$ be objects in $\ka_L$. We have a commutative diagram
\[
\begin{xy}
 \xymatrix{ \Hom((A,f), (A',g)) \ar[d] \ar[rr]^{F_L} & &\Hom((F(A),\widetilde{f}),(F(A'),\widetilde{g})) \ar[d] \\
		\Hom(A,A') \ar[rr]^F_\cong & & \Hom(F(A),F(A'))}	
\end{xy}
\]
where the vertical maps are the inclusions. We conclude that the functor $F_L$ is faithful. Let $\beta\colon F(A) \rightarrow F(A')$ be compatible with the module structures. Since $F$ is full, there exists an $\alpha$ such that $F(\alpha)=\beta$ and we have 
\[F(\alpha\circ f(l))=F(\alpha)\circ F(f(l))=F(\alpha)\circ \widetilde{f}(l)=\widetilde{g}(l)\circ F(\alpha)=F(g(l)\circ \alpha)\]
for any $l \in L$. Since $F$ is faithful, this shows $\alpha\circ f(l)=g(l)\circ \alpha$. We conclude that $\alpha$ is a morphism in $\ka_L$ and $F_L$ is full.\newline
Finally, let $(B,h)$ be an element in $\kb_L$. Using the inverse functor $F^{-1}$ we get an object $(A,f) \in \ka_L$ such that $F_L((A,f))=(B,h)$ and thus $F_L$ is essentially surjective. 
\end{proof}

\begin{ex}\label{affinecase}
Let $A$ be a $K$-algebra and let $\kc=\text{Mod}(A)$ be the abelian category of (left) $A$-modules. As one would expect, one has an equivalence 
\[\Phi\colon \text{Mod}(A)_L \cong \text{Mod}(A\otimes_K L).\]
The definition of the functor is straightforward: If $(M,f)$ is an element in $\text{Mod}(A)_L$, then one can define an $(A\otimes_K L)$-module structure on $M$ as follows:
\[ (A\otimes_K L) \times M \rightarrow M, \hspace{0.7cm} (a\otimes l, m) \mapsto f(l)(am)\]
A morphism $\alpha\colon (M,f) \rightarrow (N,g)$ is simply sent to itself, since the compatibility with the $L$-module structures translates into linearity over $A\otimes_K L$. It is obvious that $\Phi$ is faithful. It is full since for any $(A\otimes_K L)$-linear map $\varphi\colon M \rightarrow N$ one has
\[\varphi f(l)(am)=\varphi((a\otimes l) m)=(a\otimes l) \varphi(m)=g(l)(a\varphi(m))=g(l)\varphi(am)\]
and therefore $\varphi$ can be considered as a map from $(M,f)$ to $(N,g)$. Finally, the functor is essentially surjective since any $(A\otimes_K L)$-module $M$ can be considered as an $A$-module and the $L$-module structure is given by
\[L \rightarrow \End_A(M), \hspace{0.7cm} l \mapsto \left[ f(l)\colon m \mapsto \mu(1\otimes l)m \right],\]
where $\mu$ is the scalar multiplication. 

Using similar arguments one also proves $\text{QCoh}(X)_L \cong \text{QCoh}(X_L)$ for any scheme $X$ over $K$.

Note that the same argument shows that for a finite field extension the base change of the abelian category of all finitely generated $A$-modules is equivalent to the category of all finitely generated $(A\otimes_K L)$-modules. It follows that for a noetherian scheme $X$ over $K$ one has an equivalence $\text{Coh}(X)_L \cong \text{Coh}(X_L)$.
\end{ex}

The group $\Aut(L/K)$ acts on $\kc_L$ in the following way: Let $\alpha \in \Aut(L/K)$ and $(A,f) \in \kc_L$, then $\alpha(A,f):=(A, f\circ \alpha)$. If $\kc$ is equal to $\text{Mod}(A)$ for a $K$-algebra $A$, then it is easy to see that this action corresponds to the usual action of $\Aut(L/K)$ on modules. A nice property one has in this situation is

\begin{lem}\label{gal-descent}
Let $L/K$ be a finite Galois extension with Galois group $G$, let $\kc$ be a $K$-linear abelian category and $\kc_L$ the base change category. Then $G$ acts on $\kc_L$ and Galois descent holds, i.e.\ one has an equivalence between $\kc$ and the category $(\kc_L)^G$ of objects with Galois-action in $\kc_L$ together with Galois-equivariant morphisms.
\end{lem} 

\begin{proof}
By the Mitchell embedding theorem \cite{Mitchell} there exists a full exact embedding of $\kc$ into the abelian category $\text{Mod}(A)$ of modules over some $K$-algebra $A$. By Example \ref{affinecase} the category $\kc_L$ can then be embedded into $\text{Mod}(A \otimes L)$. It is classical that the pair $\text{Mod}(A)$ and $\text{Mod}(A\otimes L)$ satisfies Galois descent. Let $(M,f)$ be an object with Galois-action in $\kc_L$. By Galois descent there exists an $A$-module $N$ such that $N\otimes_A L$ is isomorphic to $(M,f)$. Considering these modules as modules over $A$ gives an isomorphism $M \cong N^d$. Since the embedding of $\kc$ into $\text{Mod}(A)$ is full and exact, this implies that $N$ is indeed an element in $\kc$ (e.g.\ because it can be written as a kernel of an endomorphism of $M$). Hence, the pair $\kc$ and $\kc_L$ satisfies Galois descent as claimed.
\end{proof}

\begin{lem}\label{bc-kom}
Let $\kc$ be a $K$-linear additive category. Then we have an equivalence: 
\[\begin{xy}
\xymatrix{\Phi\colon Kom(\kc)_L \ar[r]^\cong & Kom(\kc_L),}
\end{xy}\]
where $Kom$ denotes the category of complexes.
\end{lem}

\begin{proof}
Let $(A^\bullet= \ldots \rightarrow A^i \rightarrow A^{i+1} \rightarrow \ldots, f)$ be an object in $Kom(\kc)_L$ so that for any $l \in L$ one has a morphism of complexes $f(l)\colon A^\bullet \rightarrow A^\bullet$. For any $n \in \IZ$ the component $f(l)^n$ defines an $L$-module structure on $A^n$ and the differentials are compatible with these structures, hence are morphisms in $\kc_L$. Therefore, $A^\bullet \in Kom(\kc_L)$ and $F$ is defined on objects. A morphism $\alpha\colon (A^\bullet, f) \rightarrow (B^\bullet, g)$ is simply sent to $\alpha$ considered as a morphism of complexes in $Kom(\kc_L)$. It is now obvious that $F$ is an equivalence.  
\end{proof}

There is a forgetful functor $\Lambda\colon \kc_L \rightarrow \kc$ from the base change category to the original one, which is exact in the abelian case. It is also possible to define a functor in the other direction as follows:

If $\kc$ is a $K$-linear additive category, $V$ a $K$-vector space and $X \in \kc$ one can consider the functor
\[F^X_V\colon \kc \rightarrow Vec_K, \;\;\;\;\; C \mapsto \Hom_K(V, \Hom_\kc(X,C)).\]
This functor is representable by the object $X^{\oplus \dim_K(V)}$ which will, for obvious reasons, be denoted by $V\otimes_K X$. Here we tacitly assume that either the field extension is finite or that $\kc$ has arbitrary direct sums. Using the defining property of $V\otimes_K X$ one has an isomorphism
\[\mu\colon F^X_V(V\otimes_K X)=\Hom_K(V, \Hom_\kc(X,V\otimes_K X))\cong\End(V\otimes_K X).\]
Let us now specialize to $V=L$, where $L$ is our finite field extension. We can define an $L$-module structure on $L\otimes_K X$ as follows. Consider the element $f_0=\mu^{-1}(\id) \in \Hom_K(L, \Hom_\kc(X,L\otimes_K X))$. Any element $l \in L$ gives a $K$-linear map from $L$ to itself and therefore we can define $\alpha(l)$ to be $\mu(f_0 \circ l)$. It is easy to check that this defines a homomorphism of algebras $\alpha\colon L \rightarrow \End(L\otimes_K X)$ and thus an $L$-module structure on $L\otimes_K X$. One could equally well just use the following

\begin{lem}
Let $\ka$ be an additive category with arbitrary direct sums. There exist canonical maps ${\rm{Mat}}(I\times J, K) \rightarrow \Hom_\ka(\oplus_I X, \oplus_J X)$, where $I$ and $J$ are some index sets, which are compatible with the inclusions and projections. Via these maps, matrix multiplication corresponds to composition of maps.
\end{lem}

\begin{proof}
This is a special case of \cite[Lem.\ B3.3]{AZ}. Note that in \cite{AZ} the authors work with abelian categories, but the quoted lemma only needs the additivity.
\end{proof}

Mapping $X$ to $X\otimes_K L$ defines an exact $K$-linear functor 
\[\Xi\colon \kc \rightarrow \kc_L\]
by sending an exact sequence $X \rightarrow Y \rightarrow Z$ to its $\dim_K(V)$-fold sum. One has the 

\begin{lem}
The functor $\Xi$ is left adjoint to $\Lambda$, i.e.\ for objects $C \in \kc$ and $(D,\alpha) \in \kc_L$ one has a natural isomorphism
\[ \begin{xy} \xymatrix{\Hom_{\kc_L}(\Xi(C), (D,\alpha)) \ar[r]^\cong & \Hom_\kc(C, \Lambda(D,\alpha)).} \end{xy}\]
\end{lem}

\begin{proof}
We recall the proof from \cite{Stalder} where an inverse is  constructed as follows: Let $f$ be an element in $\Hom_\kc(C, D)$. Using $\alpha$ one defines a morphism $L \rightarrow \Hom(C,D)$ by $l \mapsto \alpha(l)\circ f$. By definition of the tensor product this corresponds to a morphism $\Xi(C) \rightarrow D$ which is compatible with the $L$-module structures. One could also just quote \cite[Prop.\ B3.16]{AZ}. Note that in \cite{AZ} the authors define tensor products in a more general setting and therefore abelian categories have to be used for some of the arguments. In our situation the additivity is in fact sufficient for the quoted statement.
\end{proof}

\begin{ex}
Consider the situation of Example \ref{affinecase}. It is easy to see that the functor $\Xi$ corresponds to tensoring an $A$-module with the ring $A\otimes_K L$ and the functor $\Lambda$ is nothing but considering a module over $A\otimes_K L$ as an $A$-module. Going from the affine situation to an arbitrary scheme $X$ over $K$ we see that $\Xi$ corresponds to $p^*$ and $\Lambda$ to $p_*$, where $p\colon X_L \rightarrow X$ is the projection. Of course, $p^*$ is exact, since $p$ is flat. Thus, the above lemma translates into the usual adjunction of the functors $p^*$ and $p_*$.  
\end{ex}

\noindent
{\bf{Convention:}} From here on we will write $p^*$ for the functor $\Xi$ and $p_*$ for $\Lambda$. 

\begin{cor}\label{inj}
If $(C,f)$ is an injective object in $\kc_L$, then $C=p_*(C,f)$ is an injective object in $\kc$. Furthermore, if $\kc_L$ has enough injective objects, then the same holds for $\kc$.
\end{cor}

\begin{proof}
The lemma above gives that the functors $\Hom(-,C)$ and $\Hom(p^*(-),(C,f))$ are isomorphic. The latter is exact, being the composition of the exact functors $p^*$ and $\Hom(-, (C,f))$. This proves the first statement. As to the second one: Consider an arbitrary element $C \in \kc$. The object $p^*(C)$ can, by assumption, be embedded into an injective object $(D,g)$. Applying the exact functor $p_*$ to this embedding we get an injection $C^{\oplus d} \rightarrow D$. Thus, $C$ can be embedded into the injective object $D$.  
\end{proof}

\begin{remark}
In fact, the converse implication of the second statement also holds, cf. \cite[Prop.\ 4.8]{LvdB}. 
\end{remark}

\begin{cor}
There is a fully faithful functor $I(\kc_L) \rightarrow I(\kc)_L$ sending $(I,f)$ to $(I,f)$ (where $I(\kc)$ resp.\ $I(\kc_L)$ denotes the category of injective objects in $\kc$ resp.\ $\kc_L$). Furthermore, $I(\kc_L)$ is closed under direct summands in $I(\kc)_L$.
\end{cor}

\begin{proof}
Only the second statement needs a proof. Let $(I,f)$ and $(J,g)$ be two elements in $I(\kc)_L$ such that their direct sum $(I\oplus J, f\oplus g)$ is in $I(\kc_L)$. Now use that a direct summand of an injective object is injective.  
\end{proof}

\section{Differential graded categories}
In this section we recall the necessary notions and facts from the theory of differential graded categories. For details see e.g.\ \cite{Drinfeld}, \cite{Keller} or \cite{LO}.

\begin{definition}
A \emph{differential graded category} or \emph{DG-category} over a field $K$ is a $K$-linear additive category $\ka$ such that for any two objects $X, Y \in \ka$ the space of morphisms $\Hom(X,Y)$ is a complex, the composition of morphisms 
\[\Hom(X,Y)\otimes \Hom(Y,Z) \rightarrow \Hom(X,Z)\]
is a chain map and the identity with respect to the composition is closed of degree $0$.
\end{definition}

\begin{ex}
The most basic example of a $K$-linear DG-category is the category of complexes of $K$-vector spaces. For two complexes $X$ and $Y$ we define $\Hom(X,Y)^n$ to be the $K$-vector space formed by families $\alpha=(\alpha^p)$ of morphisms $\alpha^p\colon X^p \rightarrow Y^{p+n}$, $p \in \IZ$. We define $\Hom_{DG}(X,Y)$ to be the graded $K$-vector space with components $\Hom(X,Y)^n$ and whose differential is given by
\[d(\alpha)=d_Y\circ \alpha -(-1)^n\alpha \circ d_X.\]
The DG-category $C_{DG}(K)$ has as objects complexes and the morphisms are defined by
\[C_{DG}(K)(X,Y)=\Hom_{DG}(X,Y).\]
Of course, starting with the category of complexes over an arbitrary $K$-linear abelian (or additive) category one can associate a DG-category to it in a similar manner.

Clearly, we get back the usual category of complexes by taking as morphisms only the closed morphisms of degree zero and we get the usual homotopy category if we replace $\Hom_{DG}(X,Y)$ by $\ker(d^0)/\im(d^{-1})$. 
\end{ex}

A \emph{DG-functor} $F\colon \ka \rightarrow \kb$ between DG-categories $\ka$ and $\kb$ is by definition required to be compatible with the structure of complexes on the spaces of morphisms. If $F, G\colon \ka \rightarrow \kb$ are two DG-functors, then we define \emph{the complex of graded morphisms} $\Hom(F, G)$ to be the complex whose $n$th component is the space formed by families of morphisms $\phi_X \in \Hom_\kb(F(X),G(X))^n$ such that $(G\alpha)(\phi_X)=(\phi_Y)(F\alpha)$ for all $\alpha \in \Hom_\ka(X,Y)$, where $X,Y \in \ka$. The differential is given by that of $\Hom_\kb(F(X),G(X))$. Using this we define the DG-category of DG-functors from $\ka$ to $\kb$, denoted by $\Hom(\ka,\kb)$, to be the category with DG-functors as objects and the above described spaces as morphisms. Note that the DG-functors between $\ka$ and $\kb$ are precisely the closed morphisms of degree zero in $\Hom(\ka,\kb)$.

To any DG-category $\ka$ one can naturally associate two other categories: Firstly, there is the \emph{graded category} $Ho^\bullet(\ka)=H^\bullet(\ka)$ having the same objects as $\ka$ and where the space of morphisms between two objects $X, Y$ is by definition the direct sum of the cohomologies of the complex $\Hom_\ka(X, Y)$. Secondly, restricting to the cohomology in degree zero we get the \emph{homotopy category} $Ho(\ka)=H^0(\ka)$. 

\begin{definition}\label{quasi}
A DG-functor $F\colon \ka \rightarrow \kb$ is \emph{quasi fully faithful} if for any two objects $X,Y$ in $\ka$ the map 
\[\Hom(X,Y) \rightarrow \Hom(F(X), F(Y))\]
is a quasi-isomorphism and $F$ is a \emph{quasi-equivalence} if in addition the induced functor $H^0(F)$ is essentially surjective. Two DG-categories $\ka$ and $\kb$ are called \emph{quasi-equivalent} if there exist DG-categories $\kc_1,\ldots, \kc_n$ and a chain of quasi-equivalences $\begin{xy}\xymatrix{\ka & \kc_1 \ar[l] \ar[r] & \cdots & \kc_n \ar[l] \ar[r] & \kb.}\end{xy}$

A DG-functor $F\colon \ka \rightarrow \kb$ is a \emph{DG-equivalence} if it is fully faithful and for every object $B \in \kb$ there is a closed isomorphism of degree 0 between $B$ and an object of $F(\ka)$.
\end{definition}

We also have to recall the following construction from \cite{BK}.

\begin{definition}
Let $\ka$ be a DG-category. Define the \emph{pretriangulated hull} $\ka^{pretr}$ of $\ka$ to be the following category. Its objects are formal expressions $(\oplus_{i=1}^n C_i[r_i],q)$, where $C_i \in \ka$, $r_i \in \IZ$, $n\geq 0$, $q=(q_{ij})$, $q_{ij} \in \Hom(C_j, C_i)[r_i-r_j]$ is homogeneous of degree 1, $q_{ij}=0$ for $i \geq j$, $dq+q^2=0$. If $C=(\oplus_{j=1}^n C_j[r_j],q)$ and $C'=(\oplus_{i=1}^m C'_i[r'_i],q')$ are objects in $\ka^{pretr}$, then the $\IZ$-graded $K$-module $\Hom(C,C')$ is the space of matrices $f=(f_{ij})$, $f_{ij} \in \Hom(C_j,C'_i)[r'_i-r_j]$ and the composition map is matrix multiplication. The differential $d\colon \Hom(C,C') \rightarrow \Hom(C,C')$ is defined by $d(f)=(df_{ij})+q'f-(-1)^lfq$ if $\deg f_{ij}=l$. The category $\ka$ is called \emph{pretriangulated} if the natural fully faithful functor $\Psi\colon \ka \rightarrow \ka^{pretr}$ is a quasi-equivalence and $\ka$ is \emph{strongly pretriangulated} if $\Psi$ is a DG-equivalence. 
\end{definition}

The reason for introducing the pretriangulated hull is that its homotopy category is always triangulated. Thus, we have the following

\begin{definition}
Let $\ka$ be a DG-category. The associated triangulated category is $\ka^{tr}:=H^0(\ka^{pretr})$.
\end{definition}


Finally we have the following fundamental notion.

\begin{definition}
Let $\kt$ be a triangulated category. An \emph{enhancement} of $\kt$ is a pair $(\kb,\epsilon)$, where $\kb$ is a pretriangulated DG-category and $\begin{xy}\xymatrix{\epsilon\colon H^0(\kb) \ar[r]^(.6)\sim & \kt}\end{xy}$ is an equivalence of triangulated categories.

The category $\kt$ is said to have a unique enhancement if it has one and for two enhancements $(\kb,\epsilon)$ and $(\kb',\epsilon')$ there exists a quasi-functor (see \cite{LO}) $\phi\colon \kb \rightarrow \kb'$ which induces an equivalence $H^0(\phi)\colon H^0(\ka)\rightarrow H^0(\kb)$. One then calls the two enhancements \emph{equivalent}. Two enhancements are called \emph{strongly equivalent} if there exists a quasi-functor $\phi$ such that $\epsilon'\circ H^0(\phi)$ and $\epsilon$ are isomorphic. 
\end{definition}

A reformulation of the above is the following: Two enhancements are identified if there exists a chain as in Definition \ref{quasi} where all the $\kc_i$ are enhancements as well.

\begin{remark}
According to (the more general) \cite[Thm.\ 9.9]{LO} the category $\Db(X)$ has a strongly unique enhancement if $X$ is a smooth and projective variety.
\end{remark}

\section{Scalar extension via DG-enhancements}

Let $\kt$ be a $K$-linear triangulated category and assume that it admits an enhancement $\ka$.

\begin{definition}
If $X$ is an object of a DG-category $\ka$, then a stucture of an $L$-module on $X$ is given by a morphism $f: L \rightarrow \End_\ka(X)$ of DG-algebras over $K$.
\end{definition}

In particular, the image of $L$ under $f$ lies in the kernel of $d^0$ of $\End_\ka(X)$.\smallskip

We thus have a category $\ka_L$ of $L$-modules. It is easy to prove the

\begin{lem}
For a DG-category $\ka$ over $K$ the category $\ka_L$ has the structure of a DG-category over $L$.
\end{lem}  

\begin{proof}
One only needs to check that the space of morphisms between two $L$-modules $(X,f)$ and $(Y,g)$ is a complex in a natural way. For this it is enough to show that for any $\alpha \in \Hom((X,f), (Y,g))$ the map $d(\alpha)$ is again in $\Hom((X,f), (Y,g))$, in other words that the differential in $\Hom(X,Y)$ restricts to the subgroup $\Hom((X,f), (Y,g))$. \newline
We know that $\alpha f(l)=g(l) \alpha$ for any $l \in L$. Differentiating both sides gives
\[d(\alpha)f(l)+\alpha d(f(l))=d(\alpha f(l))=d(g(l) \alpha)=d(g(l))\alpha+g(l)d(\alpha).\]
Since $f$ and $g$ are morphisms of DG-algebras, $d(f(l))=f(d(l))=f(0)=0$ and similarly for $g$. This completes the proof.
\end{proof}
\noindent
{\bf{Convention:}} If $f: L \rightarrow \End_\ka(X)$ and $g: L \rightarrow \End_\ka(Y)$ are two given module structures, we will sometimes write
$\Hom^{f,g}(X,Y)$ for the subcomplex $\Hom((X,f), (Y,g))$ of $\Hom(X,Y)$ defined above.

The next proposition provides a different description of the base change category.

\begin{prop}\label{equiv-def-bc}
Let $1_L$ be the $K$-linear DG-category with one object whose endomorphism ring is $L$. For a $K$-linear DG-category $\ka$ define $\ka'_L$ to be the category $\Hom(1_L, \ka)$. Then there exists an equivalence $\ka'_L \cong \ka_L$.	
\end{prop}

\begin{proof}
Let $F: 1_L \rightarrow \ka$ be a functor. It determines a unique object $X \in \ka$. Furthermore, if $F$ is a DG-functor, we get a homomorphism of DG-algebras $f: L \rightarrow \End(X)$. Thus $F$ corresponds to $(X,f)$, an $L$-module. By definition the natural transformations between two functors $F$ and $G$ correspond precisely to morphisms from $X$ to $Y$ compatible with the module structures which finishes the proof.    
\end{proof}

\begin{remark}
Let $\ka$ be a DG-category and assume that either $L/K$ is finite or that $\ka$ has arbitrary direct sums. In this situation there exists a natural DG-functor $\ka \rightarrow \ka_L$ defined as in Section 2. Using the above description it is given as the functor mapping $A \in \ka$ to the functor sending the unique object of $1_L$ to $A^{\oplus \dim_K(L)}$.  
\end{remark}

\begin{remark}
Note that there is a second possibility to associate to a $K$-linear DG-category $\ka$ an $L$-linear DG-category, namely by taking the tensor product of $\ka$ with the category $1_L$. Recall that the tensor product of two DG-categories $\ka$ and $\kb$ is defined to be the DG-category where the objects are pairs $(A,B)$ and the space of morphisms of two such pairs $(A,B)$ and $(A',B')$ is defined to be the tensor product of complexes $\Hom_\ka(A,A')\otimes \Hom_\kb(B,B')$. However, this cannot be the right construction in the geometric case, since we do not get any new objects. It rather seems that in a sense this construction corresponds to associating to $\text{Coh}(X)$ (for a scheme $X$ over $K$) the category $p^*(\text{Coh}(X))$, where $p: X_L \rightarrow X$ is the projection. 
\end{remark}

\begin{definition}
Let $\kt=H^0(\ka)$ be the homotopy category of a pretriangulated $K$-linear DG-category $\ka$ and let $L/K$ be a field extension. We define the \emph{base change category} $\kt_L$ to be the smallest thick (i.e.\ closed under taking direct summands) full triangulated subcategory of $H^0((\ka_L)^{pretr})$ containing the image of $\kt$ under the functor induced by $\ka \rightarrow \ka_L \hookrightarrow (\ka_L)^{pretr}$.
\end{definition}

\begin{remark}
In the above definition and in the following we tacitly assume that $\ka$ has infinite direct sums or that $L/K$ is finite. 
\end{remark}

Clearly one would like to see that this definition does not depend on the enhancement. Unfortunately we were not able to prove this statement. This problem seems to be related to the fact that the internal $\Hom$-functor (which we use, cf.\ Proposition \ref{equiv-def-bc}) in the 2-category of DG-categories does not respect quasi-equivalences (see also Remark \ref{toens-rhom}). However, there is the following partial result. Note that a $K$-linear functor between two $K$-linear DG-categories $\ka$ and $\kb$ induces an $L$-linear functor from $\ka_L$ to $\kb_L$.

\begin{prop}\label{ind-choice}
Let $\ka$ and $\kb$ be two $K$-linear pretriangulated DG-categories, consider $\ka_L$ and $\kb_L$ and let $\Phi\colon \ka\rightarrow \kb$ be a quasi-equivalence. Write $\kt\cong H^0(\ka)\cong H^0(\kb)$. Assume that for all $(A,f), (A',g) \in \ka_L$ and for all $\varphi \in \Hom_\ka(A,A')$ we have that $d\varphi f=g d\varphi$ implies that $\varphi f=g\varphi$ (we call this condition $(\ast)$) and similarly for $\kb$. 
Furthermore, assume that the categories $H^\bullet(\ka)$ and $H^\bullet(\kb)$ are of finite type, that is, the morphism spaces are finite-dimensional. 
Then $\Phi^G$ is quasi fully faithful.

If in addition to the above assumptions there also exists an adjoint quasi-equivalence $\Psi$, then $H^0(\ka)_L$ and $H^0(\kb)_L$ are equivalent.
\end{prop}

\begin{proof}
Since the reasoning is very similar to the proof of \cite[Lem.\ 3.11, Prop.\ 3.12]{PS1}, we refer the reader to that paper. The proof there gives that $\ka_L$ and $\kb_L$ are quasi-equivalent, hence also the pretriangulated hulls and therefore also the respective triangulated categories generated by the image of $\kt$.
\end{proof}

Let us now consider the results this construction produces in some standard examples. 

\begin{prop}\label{smooth-bc}
Let $X$ be a smooth projective variety over $K$ and consider $\Db(X)$. Then $(\Db(X))_L$ is equivalent to $\Db(X_L)$. A similar result holds for the bounded derived category of quasi-coherent sheaves.
\end{prop}

\begin{proof}
We know that $\kt \cong \widetilde{K}^+(I(\kc))$, where $I(\kc)$ is the additive category of injective objects in $\kc={\rm{QCoh}}(X)$ and $\widetilde{K}^+(I(\kc))$ is the homotopy category of bounded-below complexes of injectives having only finitely many coherent cohomology objects. It is well-known that the DG-category of bounded-below complexes of injective objects with bounded coherent cohomology $\ka=C_{DG}(I(\kc))$ is an enhancement of $\kt$. Base change for this DG-category produces $C_{DG}(I(\kc)_L)$, which is a pretriangulated DG-category and therefore taking the pretriangulated hull does not change its homotopy category. Using Corollary \ref{inj} it is easy to see that $I(\kc_L)$ can be embedded as a full thick subcategory into $I(\kc)_L$ and hence $\Db(X_L)=\widetilde{K}^+(I(\kc_L))$ is a full triangulated thick subcategory in $\widetilde{K}^+(I(\kc)_L)=Ho(\ka_L)$ (where $\widetilde{K}^+$ is defined similarly as above). Clearly, $\Db(X_L)$ contains $\Db(X)$. In fact, $\Db(X_L)$ is the smallest thick triangulated subcategory of $Ho(\ka_L)$ with this property: In \cite{Orlov} it is shown that the category $\Db(X_L)$ has a classical generator, i.e.\ an object $E$ with the property that the smallest triangulated thick subcategory of $\Db(X_L)$ containing $E$ is everything. Now use that the classical generator $E$ is a direct sum of tensor powers of the very ample line bundle and therefore is in the image of the functor $\Db(X) \rightarrow \Db(X_L)$. The reasoning in the quasi-coherent case is similar.
\end{proof}

\begin{remark}
There exists a different enhancement of $\Db(X)$ if $X$ is smooth and projective. We will need some notation: Denote by $\kc(X)$ the pretriangulated DG-category consisting of bounded-below complexes of $\ko_X$-modules with bounded coherent cohomology. Now, we know that $\Db(X)$ is equivalent to the category $\Perf(X)$ of perfect complexes, that is, finite complexes of vector bundles. Choosing a finite affine covering $\mathcal{U}$ of $X$, one has the (strongly) pretriangulated DG category $\kp(\ku) \subset \kc(X)$ which, by definition, is the smallest full DG-subcategory of $\kc(X)$ containing all \v Cech resolutions of elements of $\Perf(X)$ and closed under taking cones of closed morphisms of degree zero. This category is an enhancement of $\Db(X)$ by \cite[Lem.\ 6.7]{BLL}. It is easy to see that the category $\kp(\ku)_L$ is equivalent to $\kp(\ku_L)$ (where $\ku_L$ is the affine covering of $X_L$ given by pulling back $\ku$) and hence its homotopy category is equivalent to $\Db(X_L)$. Using the same arguments as above one sees that our definition produces the expected result if one works with this enhancement.  
\end{remark}

There is a slight variation of the above result.

\begin{prop}\label{noeth-bc}
If $X$ is a noetherian scheme over $K$ and $L/K$ is a finite Galois extension, then $(\rm{D}^\ast(X))_L\cong \rm{D}^\ast(X_L)$, where $\ast=\text{b},+,-,\emptyset$. A similar result holds for the derived category of quasi-coherent sheaves.
\end{prop}

\begin{proof}
As in the previous proposition one shows that $\Db(X_L)$ contains $\Db(X)$. To show that $\Db(X_L)$ is indeed the smallest thick triangulated subcategory of $Ho(\ka_L)$ (notation as before) one uses the formula
\[p^*p_*(E)=\sum_{g \in G} g^*(E),\]
where $p: X_L \rightarrow X$ is the projection and $G$ is the Galois group.
\end{proof}

We also have the following result in the non-geometric situation:

\begin{prop}\label{abcat-bc}
Let $\kc$ be an abelian category with enough injectives and with generators (for details see e.g.\ \cite[Ch.\ II, 15]{Mitchell2}). Then $\Db(\kc)_L$ is equivalent to $\Db(\kc_L)$.
\end{prop}

\begin{proof}
One uses enhancements by injective objects and the fact that if $(C_i)_{i \in I}$ is a set of generators for $\kc$, then $(\Xi(C_i))_{i \in I}$ is a set of generators for $\kc_L$, cf.\ \cite[Prop.\ 4.8]{LvdB}. 
\end{proof}

\begin{remark}\label{toens-rhom}
One of the main results of \cite{Toen} is the construction of an internal Hom-functor $\RHom$ in the homotopy category of DG-categories, that is, the 2-category obtained by localisation with respect to quasi-equivalences. Thus, a possible enhancement-independent definition of a scalar extension could be given by $\RHom(1_L, \ka)$, where $\ka$ is any enhancement of the given triangulated category $\kt$. However, we know that
\[\RHom_c(\Db_{DG}(X), \Db_{DG}(Y))\cong \Db_{DG}(X\times Y)\]
for $X$, $Y$ smooth projective, where $\RHom_c$ is the category of so-called continuous functors and $\Db_{DG}$ denotes an enhancement of the bounded derived category (see \cite[Sect.\ 8]{Toen} for details). On the other hand, a similar result also holds for the unbounded categories. Hence it seems that $\RHom(1_L, \ka)$ might not be the right thing to do. 
\end{remark}

\begin{remark}
Let $\kc$ be an abelian category. It is an interesting question whether one could actually define the base change category of $\Db(\kc)$ (or $K^b(\kc)$) simply as $Ho((\ka_L)^{pretr})$ for an enhancement $\ka$ of $\Db(\kc)$. Let us investigate the general case: Let $\kc$ be an abelian category with enough injectives and consider $\kt=\Db(\kc)$. Then $Ho((\ka_L)^{pretr})=\widetilde{K}^+(I(\kc)_L)$, where the latter denotes bounded-below complexes of objects in $I(\kc)_L$ with finitely many cohomology objects. The proof of the statement $Ho((\ka_L)^{pretr})=\Db(\kc_L)$ boils down to proving the equivalence $\widetilde{K}^+(I(\kc)_L) \cong \widetilde{K}^+(I(\kc_L))$, where of course $I(\kc_L)$ denotes the category of injective objects in $\kc_L$ and hence $\Db(\kc_L) \cong \widetilde{K}^+(I(\kc_L))$. As above it is easy to see that $\widetilde{K}^+(I(\kc_L))$ is a full triangulated subcategory in $\widetilde{K}^+(I(\kc)_L)$. In order to prove that the embedding is essentially surjective, one would in particular have to show that for any injective object $I \in  I(\kc)$ and any module structure $f$ the object $(I,f)$ is in the image. This reduces to the statement that $(I,f)$ is isomorphic to an injective object in $\kc_L$. Thus, one has to show the equality $I(\kc)_L \cong I(\kc_L)$. It is unclear under which conditions this can be proved.  
\end{remark}



We conclude this section by sketching a different approach towards the definition of base change which was suggested to us by Prof.\ V.\ A.\ Lunts. Following Keller one calls a triangulated category $\kt$ \emph{algebraic} if it is the stable category of a Frobenius exact category (for the definition of the latter see \cite[Ch.\ IV.3, Ex.\ 4-8]{GeMa}). There is a close connection between algebraic triangulated categories and derived categories of DG-categories (see \cite{Keller2}). We illustrate it in a special case: Namely, by a result of Rouquier \cite{Rouquier} the derived category $\Db(X)$ of a quasi-projective scheme $X$ over a perfect field $K$ is equivalent to $\Perf(A)$, the category of perfect complexes over a DG-algebra $A$, i.e.\ the smallest thick subcategory of the derived category of $A$ containing $A$. Here, $A$ is determined by a strong generator (see Definition \ref{str-gen} in the next section) $E$ of $\Db(X)$. To be more precise, $A=\RHom(E,E)$. One could simply define the base change category $\Db(X)_L$ as $\Perf(A\otimes_K L)$. Let us check that this gives the wanted result: We can assume $E$ to consist of injective objects and therefore $\RHom(E,E)$ is just the complex $\Hom_{DG}(E,E)$. Then by \cite[Thm.\ 2.1.2 and Lem.\ 3.4.1]{BonvdB} the object $p^*(E)$ is a generator of $\Db(X_L)$ (here we tacitly assume that the field extension is finite since we need $\Spec(L)$ to be a scheme of finite type over $\Spec(K)$). Hence $\Db(X_L)\cong \Perf(B)$, where $B=\RHom(p^*(E),p^*(E))=\Hom_{DG}(p^*(E),p^*(E))$, where the second equality holds because the pullback of an injective sheaf is injective. But
\[B=\Hom_{DG}(p^*(E),p^*(E))=\Hom_{DG}(E,E)\otimes_K L=A\otimes L.\]
Hence $\Db(X_L)\cong \Perf(A\otimes_K L)$.

If one chooses a different generator $E'$, the same proof shows that $\Perf(A'\otimes_K L)$ is again equivalent to $\Db(X_L)$. Here, the choice of an enhancement is somewhat hidden, but it is indeed present, because we need the DG-structure to define the DG-algebra $A$. 

This definition is certainly more elegant and one could apply it to a vast class of examples, since most triangulated categories arising in algebraic geometry (and representation theory) are in fact algebraic. In the general case one does not have an equivalence between $\kt$ and the category of perfect complexes over some DG-algebra, but $\kt$ is rather equivalent to (a full subcategory of) the derived category $\text{D}(\ka)$ of some DG-category $\ka$ (for the definition of $\text{D}(\ka)$ see \cite{Keller}). For the last statement one has to impose some conditions on $\kt$. One could then define the base change category as (a certain subcategory of) the derived category of $\ka\otimes_K 1_L$. The disadvantage of this approach is that the DG-algebras resp.\ DG-categories that appear are in general very difficult to describe.  



\section{Dimension under scalar extensions}

In \cite{Rouquier} the dimension of a triangulated category was introduced. To recall the definition we need some notation. If $\ki$ is a subcategory of a triangulated category $\kt$, then $\left\langle \ki\right\rangle$ denotes the smallest full subcategory of $\kt$ which contains $\ki$ and is closed under isomorphisms, finite direct sums, direct summands and shifts. If $\ki_1$ and $\ki_2$ are two subcategories, then $\ki_1 \ast \ki_2$ is the full subcategory of objects $M$ in $\kt$ such that there exists a triangle $\begin{xy}\xymatrix{I_1 \ar[r] & M \ar[r] & I_2}\end{xy}$ with $I_i \in \ki_i$. We also define $\ki_1 \diamond \ki_2=\left\langle \ki_1 \ast \ki_2 \right\rangle$ and set inductively $\left\langle\ \ki \right\rangle_k=\left\langle \ki \right\rangle_{k-1} \diamond \left\langle \ki \right\rangle$. If $\ki$ consists of one object $E$ we denote $\left\langle \ki \right\rangle$ by $\left\langle E \right\rangle_1$ and set $\left\langle E \right\rangle_k=\left\langle E 
 \right\rangle_{k-1} \diamond \left\langle E \right\rangle_1$.

\begin{definition}\label{str-gen}
An object $E$ of $\kt$ is called a \emph{strong generator} if $\left\langle E \right\rangle_n=\kt$ for some $n \in \IN$. The \emph{dimension} of $\kt$ is the smallest integer $d$ such that there exists an object $E$ with $\left\langle E \right\rangle_{d+1}=\kt$. The \emph{dimension spectrum} of $\kt$ is the set of all integers $k$ such that there exists an $E$ with the property that $\left\langle E \right\rangle_{k+1}=\kt$ but $\left\langle E \right\rangle_k \neq \kt$.
\end{definition}

\begin{ex}
For a smooth affine scheme Rouquier showed that $\dim(X)=\dim(\Db(X))$ (cf.\ \cite{Rouquier}). In \cite{Orlov} Orlov showed that the dimension of the bounded derived category of a smooth projective curve $C$ of genus $g \geq 1$ is $1$ and conjectured that for any smooth quasi-projective variety $X$ the equality $\dim(\Db(X))=\dim(X)$ holds. In \cite{BF} the conjecture was verified for triangulated categories possessing a so called tilting object, which is true e.g.\ for the derived categories of del Pezzo surfaces with Picard rank at most seven and Hirzebruch surfaces. 
\end{ex}

We will now investigate the following natural question: How does the dimension of a triangulated category behave under scalar extensions?

\begin{prop}\label{dim-ab-bc}
Let $\kc$ be an abelian category with enough injectives and generators and let $L/K$ be a finite Galois extension. Assume that the dimension of $\Db(\kc)$ is finite. Then $\dim(\Db(\kc)_L)= \dim(\Db(\kc))$.
\end{prop}

\begin{proof}
We know that $\Db(\kc)_L \cong \Db(\kc_L)$ (Proposition \ref{abcat-bc}). By Lemma \ref{gal-descent} the category $\Db(\kc)$ is dense in $\Db(\kc_L)$, since for any object $A \in \Db(\kc_L)$ the object $\oplus_{g \in G} g^*(A)$ is invariant under the Galois action and hence is isomorphic to an object of $\Db(\kc)$. Recall that we write $p^*$ for the functor $\Xi$ of Section 3.1. If $\left\langle E \right\rangle^{\Db(\kc)}_n=\Db(\kc)$ for some $E$ in $\Db(\kc)$, then, by the above argument, $\left\langle p^*(E) \right\rangle^{\Db(\kc_L)}_n=\Db(\kc_L)$. This gives the inequality ``$\leq$''. 

For the converse consider a strong generator $F$ in $\Db(\kc_L)$ and denote the dimension of $\Db(\kc_L)$ by $n$. Assume that an object $M \in \Db(\kc)$ can be reached from $F$ in one step, i.e.\ that there exists a triangle
\[\begin{xy}\xymatrix{ F[-1] \ar[r] & F \ar[r] & p^*(M) \ar[r] & F.}\end{xy}\]
Applying all $g \in G$ to this triangle, taking the direct sum and denoting the object $\oplus_{g \in G}g^*(F)$ by $\widetilde{E}$ gives the triangle
\[\begin{xy}\xymatrix{ \widetilde{E}[-1] \ar[r] & \widetilde{E} \ar[r] & \oplus_{g \in G} g^*(p^*(M))=p^*(M)^{\oplus d} \ar[r] &\widetilde{E}.}\end{xy}\]
By Galois descent this is a triangle in $\Db(\kc)$ and, therefore, the object $M^{\oplus d}$ can be built from $E$, where $p^*(E)=\widetilde{E}$, in one step. Induction on the number of steps gives the inequality ``$\geq$''. 
\end{proof}

The same arguments also give
\begin{prop}\label{dim-noe-bc}
Let $X$ be a smooth projective variety over $K$ and $L/K$ be a finite Galois extension. If $\dim(\Db(X))$ is finite, then $\dim(\Db(X))=\dim(\Db(X_L))$. In particular, if $\dim(\Db(X))=\dim(X)$, then $\dim(\Db(X_L))=\dim(X_L)=\dim(X)$ for any finite Galois extension.\qqed
\end{prop}

\begin{remark}
In the proof of Proposition \ref{dim-ab-bc} we have seen that for a finite Galois extension the category $p^*(\Db(\kc))$ is dense in $\Db(\kc_L)$. One could try and use this to define scalar extension without enhancements as follows. Let $\kt$ be a triangulated category over $K$ and let $L/K$ be a finite Galois extension. Set $\kt'$ to be the additive category $\kt\otimes L$: The objects in this category are the same as in $\kt$ and for any two objects $T$, $T'$ we set
\[\Hom_{\kt'}(T,T'):=\Hom_\kt (T,T')\otimes_K L.\]
Now consider the fully faithul Yoneda embedding of $\kt'$ into the category $\kb:=\text{Fun}(\kt',L-\text{Vec})$ of additive functors from $\kt'$ to the category of $L$-vector spaces. Denote by $\widetilde{\kt'}$ the closure of $\kt'$ in $\kb$ under direct summands, that is, the objects of $\widetilde{\kt'}$ are those objects $B$ in $\kb$ such that there exists a $B' \in \kb$ with $B\oplus B' \in \kt'$. If $\kt=\Db(\kc)$ as above, then it is clear that $\kt'$ ican be identified with $p^*(\Db(\kc))$ and hence $\widetilde{\kt'}$ is equivalent to $\Db(\kc_L)$. However, it is difficult to see how one can define a triangulated structure on $\widetilde{\kt'}$ in general. 
\end{remark}

\end{document}